\documentclass[11pt,leqno]{amsart}
\usepackage{amscd}
\usepackage[a4paper,centering,textwidth=125mm,textheight=195mm]{geometry}

\newcommand{\pol}[1]{\mathfrak{b}_{#1}}
\newcommand{\Wind}{\mathfrak{i}_0}

\DeclareMathOperator{\rad}{rad}
\DeclareMathOperator{\car}{char}
\DeclareMathOperator{\Cor}{Cor}
\DeclareMathOperator{\Trd}{Trd}
\DeclareMathOperator{\Nrd}{Nrd}
\DeclareMathOperator{\Br}{Br}
\newcommand{\id}{\mathsf{id}}

\swapnumbers
\numberwithin{equation}{section}
\newtheorem{thm}[equation]{Theorem}
\newtheorem*{thm*}{Theorem}

\newtheorem{lem}[equation]{Lemma}

\theoremstyle{definition}

\newtheorem*{rems}{Remarks}

\theoremstyle{plain}

\title[Transfer over quadratic extensions]{Transfer
  of quadratic forms and of quaternion algebras over
  quadratic field extensions}  
\author[K.J. Becher]{Karim Johannes Becher}
\author[N. Grenier-Boley]{Nicolas Grenier-Boley} 
\author[J.-P. Tignol]{Jean-Pierre Tignol}
\address{Universiteit Antwerpen, Departement Wiskunde en Informatica, Middelheimlaan~1, B-2020 Antwerpen, Belgium}
\email{KarimJohannes.Becher@uantwerpen.be}
\address{Universit\'e de Rouen Normandie, Laboratoire de
  Didactique Andr\'e Revuz\linebreak (LDAR-EA 4434), F-76130 Mont-Saint-Aignan, France.}
\email{nicolas.grenier-boley@univ-rouen.fr}
\address{Universit\'e catholique de Louvain, ICTEAM Institute, Avenue
  G.~Lema\^{\i}tre 4, Box~L4.05.01,
B-1348 Louvain-La-Neuve, Belgium.}
\email{jean-pierre.tignol@uclouvain.be}
\thanks{The first author was supported by the FWO Odysseus Programme
  (project \emph{Explicit Methods in Quadratic Form Theory}), funded by the
  Fonds Wetenschappelijk Onderzoek -- Vlaanderen. The third author
  acknowledges support from the Fonds de la Recherche
  Scientifique--FNRS under grant n$^\circ$~J.0014.15.} 
\begin{document}
\maketitle

A well-known theorem of Albert states that if a tensor product of two
quaternion division algebras $Q_1$, $Q_2$ over a field $F$ of
characteristic different from~$2$ is not a division algebra, then
there exists a quadratic extension $L$ of $F$ that embeds as a
subfield in $Q_1$ and in $Q_2$; see \cite[(16.29)]{BOI}. The same
property holds in characteristic~$2$, with the additional condition
that $L/F$ is separable: this was proved by Draxl \cite{Dxl}, and
several proofs have been proposed: see \cite[Th.~98.19]{EKM}, and
\cite{Tits} for a list of earlier references.

Our purpose in this note is to extend the Albert--Draxl Theorem by
substituting for the tensor product of two quaternion algebras the
corestriction of a single quaternion algebra over a quadratic
extension. Our main result is the following:

\begin{thm}
  \label{thm:main}
  Let $F$ be an arbitrary field and let $K$ be a quadratic \'etale
  $F$-algebra.\footnote{This means that $K$ is either a quadratic
    separable field extension of $F$, or $K\simeq F\times F$.} For
  every quaternion $K$-algebra $Q$, the following conditions are
  equivalent:
  \begin{enumerate}
  \item[(i)]
  $Q$ contains a quadratic $F$-algebra linearly disjoint from $K$;
  \item[(ii)]
  $Q$ contains a quadratic \'etale $F$-algebra linearly disjoint
  from~$K$;
  \item[(iii)]
  $\Cor_{K/F}Q$ is not a division algebra.
  \end{enumerate}
\end{thm}

Note that when $K=F\times F$ the quaternion $K$-algebra $Q$ has the
form $Q_1\times Q_2$ for some quaternion $F$-algebras $Q_1$, $Q_2$,
and $\Cor_{K/F}Q=Q_1\otimes_FQ_2$. Thus, in this particular case
Theorem~\ref{thm:main} is equivalent to the Albert--Draxl Theorem. The
more general case is needed for the proof of the main result in
\cite{BGBT}. 

If the characteristic $\car F$ is different from~$2$,
Theorem~\ref{thm:main} is proved in \cite[(16.28)]{BOI}. The proof
below is close to that in \cite{BOI}, but it does not require any
restriction on the characteristic. The idea is to use a transfer
of the norm form $n_Q$ of $Q$ to obtain an Albert form of $\Cor_{K/F}Q$,
which allows us to substitute for~(iii) the condition that the
transfer of $n_Q$ has Witt index at least~$2$. To complete the
argument, we need to relate totally isotropic subspaces of the
transfer to subforms of $n_Q$ defined over $F$. This is slightly more
delicate in characteristic~$2$. Therefore, we first discuss the
transfer of quadratic forms in \S\ref{sec:trans}, and give the proof
of Theorem~\ref{thm:main} in \S\ref{sec:proof}. In the last section,
we sketch an alternative proof of Theorem~\ref{thm:main} based on a
proof of the Albert--Draxl Theorem due to Knus \cite{Knus}. This
alternative proof relies on an explicit construction of an Albert form
for the corestriction of a quaternion algebra.

\subsection*{Notations and Terminology}
\emph{Quaternion algebras} over an arbitrary field $F$ are
$F$-algebras obtained from an \'etale quadratic $F$-algebra $E$ and an
element $a\in F^\times$ by the following construction:
\[
(E/F,a)=E\oplus Ez
\]
with multiplication defined by the equations
\[
z^2=a\qquad\text{and}\qquad
z\ell=\iota(\ell)z\quad\text{for $\ell\in L$,}
\]
where $\iota$ is the nontrivial automorphism of $L$.

For quadratic and bilinear forms, we generally follow the notation and
terminology of \cite{EKM}. Thus, if $\varphi\colon V\to F$ is a
quadratic form on a (finite-dimensional) vector space $V$ over an
arbitrary field $F$, we let $\pol{\varphi}\colon V\times V\to F$
denote the polar form of $\varphi$, defined by
\[
\pol{\varphi}(x,y)=\varphi(x+y)-\varphi(x)-\varphi(y) \qquad\text{for
  $x$, $y\in V$.}
\]
We set
\begin{eqnarray*}\rad\pol\varphi & = & \{x\in V\mid
        \pol\varphi(x,y)=0\mbox{ for all } y\in V\}\\ 
\rad\varphi & = & \{x\in \rad\pol\varphi\mid \varphi(x)=0\}\,
\end{eqnarray*}
and observe that these sets are $F$-subspaces of $V$ with
$\rad\varphi\subseteq \rad\pol\varphi$.  
Moreover,
if $\car F\neq 2$ then $\varphi(x)=\frac{1}{2}\pol\varphi(x,x)$ for
all $x\in V$ and thus $\rad\varphi=\rad\pol\varphi$. 
We call the quadratic form $\varphi$
\emph{nonsingular}\footnote{Nonsingular quadratic forms are not
  defined in \cite{EKM}.} if
$\rad\pol\varphi=\{0\}$, \emph{regular} if $\rad\varphi=\{0\}$ and
\emph{nondegenerate} if $\varphi_K$ is regular for every field
extension $K/F$ or equivalently (by \cite[Lemma~7.16]{EKM}) if
$\varphi$ is regular and $\dim_F\rad\pol\varphi\leq 1$. 
Thus, every nonsingular form is nondegenerate and every nondegenerate
form is regular; moreover, all three conditions are equivalent in the
case where $\car F\neq 2$. 

The \emph{Witt index} of a quadratic form $\varphi$ on a vector space
$V$ is the dimension of the maximal totally isotropic subspaces of
$V$, i.e., the maximal subspaces $U\subseteq V$ such that
$\varphi(u)=0$ for all $u\in U$; see \cite[Prop.~8.11]{EKM}. We write
$\Wind(\varphi)$ for the Witt index of $\varphi$.

We will need the following easy observation:

\begin{lem}
  \label{L:L}
  Let $\varphi$ be a regular quadratic form on a vector space $V$. If
  $\varphi$ is isotropic, the isotropic vectors span $V$.
\end{lem}

\begin{proof}
  Let $V_0\subseteq V$ be the subspace spanned by the isotropic
  vectors of $V$, and let $v\in V\setminus \{0\}$ be an isotropic
  vector. Since $\rad(\varphi)=\{0\}$ there exists $w\in V$ 
  such that $\pol\varphi(v,w)=1$. If $x\in V$ is such that
  $\pol\varphi(v,x)\neq0$, then the vector
  $x-\varphi(x)\pol\varphi(v,x)^{-1}v$ is isotropic, hence it belongs
  to $V_0$. It follows that $x\in V_0$ since $v\in V_0$. Thus, $V_0$
  contains all the vectors that are not orthogonal to $v$. In
  particular, it contains $w$. If $x\in V$ is orthogonal to $v$, then
  $\pol\varphi(v,x+w)=1$, hence $x+w\in V_0$, and therefore $x\in V_0$
  since $w\in V_0$. Thus, $V_0=V$.
\end{proof}

\section{Isotropic transfers}
\label{sec:trans}

Let $F$ be an arbitrary field and let $K$ be a quadratic field
extension of $F$. Fix a nonzero $F$-linear functional $s\colon K\to F$
such that $s(1)=0$. For every quadratic form $\varphi\colon V\to K$ on a
$K$-vector space, the transfer $s_*\varphi$ is the quadratic form on
$V$ (viewed as an $F$-vector space) defined by
\[
s_*\varphi(x)=s\bigl(\varphi(x)\bigr)\qquad\text{for $x\in V$.}
\]
If $\varphi$ is nonsingular, then $s_*\varphi$ is nonsingular: see
\cite[Lemma~20.4]{EKM}. For every quadratic form $\psi$ over $F$, we
let $\psi_K$ denote the quadratic form over $K$ obtained from $\psi$
by extending scalars to $K$.

The following result is well-known in characteristic different from
$2$, but it appears to be new in characteristic $2$. 

\begin{thm}\label{P:}
  Let $\varphi$ be a nonsingular quadratic form over $K$.
  Then there exists a nondegenerate quadratic form $\psi$ over $F$
  with $\dim \psi=\Wind(s_\ast\varphi)$ such that $\psi_K$ is a
  subform of $\varphi$. 
\end{thm}

\begin{proof}
  Substituting for $\varphi$ its anisotropic part, we may assume
  $\varphi$ is anisotropic. We may also assume
  $\Wind(s_*\varphi)\geq1$, for otherwise there is nothing to
  show. Pick an isotropic vector $u\in V$ for the form $s_*\varphi$;
  we thus have $\varphi(u)\in F$ and $\varphi(u)\neq0$ since $\varphi$
  is anisotropic. If $\Wind(s_*\varphi)=1$, then we may choose
  $\psi=\varphi\rvert_{Fv}$. For the rest of the proof, we assume
  $\Wind(s_*\varphi)\geq2$, and we argue by induction on
  $\Wind(s_*\varphi)$. 

  Since $\varphi$ is nonsingular, we may find $v\in V$ such that
  $\pol\varphi(u,v)=1$. Let $\lambda\in K$ be such that
  $s(\lambda)=1$. We have $s_*\varphi(u)=0$ and
  \[
  \pol{s_*\varphi}(u,\lambda v)=s\bigl(\pol\varphi(u,\lambda v)\bigr)
  = s(\lambda)=1,
  \]
  hence the restriction of $s_*\varphi$ to the $F$-subspace
  $U\subseteq V$ spanned by $u$ and $\lambda v$ is nonsingular and
  isotropic. Therefore, $s_*\varphi\rvert_U$ is hyperbolic. Let
  $W=U^\perp\subseteq V$ be the orthogonal complement of $U$ in $V$
  for the form $s_*\varphi$. Since $\Wind(s_*\varphi)\geq2$, the form
  $s_*\varphi\rvert_W$ is isotropic. By Lemma~\ref{L:L} we may find an
  isotropic vector $w\in W$ such that
  $\pol\varphi(u,w)\neq0$. However, $\pol{s_*\varphi}(u,w)=0$ since
  $w\in W=U^\perp$; therefore $\pol\varphi(u,w)\in F^\times$. Moreover
  $\varphi(w)\in F$ since $w$ is isotropic for
  $s_*\varphi$. Therefore, the restriction of $\varphi$ to the
  $F$-subspace of $V$ spanned by $u$ and $w$ is a quadratic form
  $\psi_1$ over $F$. 

  Observe that $u$ and $w$ are $K$-linearly independent: if $w=\alpha
  u$ with $\alpha\in K$, then
  \[
  \pol\varphi(u,w)=\alpha \pol\varphi(u,u) = 2\alpha\varphi(u).
  \]
  Since $\pol\varphi(u,w)\in F^\times$ and $\varphi(u)\in F$, it
  follows that $\alpha\in F$, which is impossible. Therefore,
  $\psi_{1K}$ is a $2$-dimensional subform of $\varphi$, and $\psi_1$
  is nonsingular because $\varphi$ is anisotropic. Since $s_*(\psi_{1K})$ is
  hyperbolic, for the orthogonal complement $\varphi'$ of $\psi_{1K}$
  in $\varphi$ we obtain $\Wind(s_*\varphi')=\Wind(s_*\varphi)-2$. The
  theorem follows by induction.
\end{proof}

\begin{rems}
  1. If $s_*\varphi$ is hyperbolic, then
  $\Wind(s_*\varphi)=\dim\varphi$, hence Theorem~\ref{P:} shows that
  $\varphi=\psi_K$ for some quadratic form $\psi$ over $F$. This
  particular case of Theorem~\ref{P:} is established
  in~\cite[Th.~34.9]{EKM}. 

  2. If $\car F=2$ and $\Wind(s_*\varphi)$ is odd, then the quadratic
  form $\psi$ in Theorem~\ref{P:} cannot be nonsingular, since
  nonsingular quadratic forms in characteristic~$2$ are
  even-dimensional; in particular, $\psi_K$ is not an orthogonal
  direct summand of $\varphi$. 

  3. If the extension $K/F$ is purely inseparable, then
  $\Wind(s_*\varphi)$ is necessarily even. This follows because the
  $K$-subspace spanned by each isotropic vector for $s_*\varphi$ is a
  $2$-dimensional $F$-subspace that is totally isotropic for
  $s_*\varphi$. 

  4. The analogue of Theorem~\ref{P:} for symmetric bilinear forms is
  proved in \cite[Prop.~34.1]{EKM}.
\end{rems}

\section{Proof of Theorem~\ref{thm:main}}
\label{sec:proof}

Let $F$ be an arbitrary field. As in \cite{EKM}, we write $I_q(F)$ for
the Witt group of nonsingular quadratic forms of even dimension over
$F$, and $I(F)$ for the ideal of even-dimensional forms in the Witt
ring $W(F)$ of nondegenerate symmetric bilinear forms over $F$, and we
let $I^n_q(F)=I^{n-1}(F)I_q(F)$ for $n\geq2$. Let also $\Br_2(F)$
denote the $2$-torsion subgroup of the Brauer group of $F$. Recall
from~\cite[Th.~14.3]{EKM} the group homomorphism
\[
e_2\colon I^2_q(F)\to \Br_2(F)
\]
defined by mapping the Witt class of a quadratic form $\varphi$ to the
Brauer class of its Clifford algebra.

\begin{lem}
  Let $K$ be a quadratic field extension of an arbitrary field $F$,
  and let $s\colon K\to F$ be a nonzero $F$-linear functional such
  that $s(1)=0$. The following diagram is commutative:
  \[
  \begin{CD}
    I^2_q(K) @>{s_*}>> I^2_q(F)\\
    @V{e_2}VV @VV{e_2}V\\
    \Br_2(K) @>{\Cor_{K/F}}>> \Br_2(F).
  \end{CD}
  \]
\end{lem}

Following the definition in \cite[Rem.~6.9.2]{GS},
$\Cor_{K/F}\colon\Br_2(K)\to\Br_2(F)$ is the zero map if $K$ is a
purely inseparable extension of $F$.

\begin{proof}
  We have $I^2_q(K)=I(F)I_q(K) + I(K)I_q(F)$ by
  \cite[Lemma~34.16]{EKM}, hence $I^2_q(K)$ is generated by Witt
  classes of $2$-fold Pfister forms that have a slot in
  $F$. Commutativity of the diagram follows by Frobenius reciprocity
  \cite[Prop.~20.2]{EKM}, the computation of transfers of $1$-fold
  Pfister forms in \cite[Lemma~34.14]{EKM} and \cite[Cor.~34.19]{EKM},
  and the projection formula in cohomology \cite[Prop.~3.4.10]{GS}.
\end{proof} 

\begin{proof}[Proof of Theorem~\ref{thm:main}]
  Since (ii)~$\Rightarrow$~(i) is clear, it suffices to prove
  (i)~$\Rightarrow$~(iii) and (iii)~$\Rightarrow$~(ii).

  If (i) holds, then we may represent $Q$ in the form $(LK/K,b)$ where
  $L$ is a quadratic \'etale $F$-algebra linearly disjoint from~$K$
  and $b\in K^\times$, or in the form $(M/K,b)$ where $M$ is a
  quadratic \'etale $K$-algebra and $b\in F^\times$. In each case the
  projection formula \cite[Prop.~3.4.10]{GS} shows that $\Cor_{K/F}Q$
  is Brauer-equivalent to a quaternion algebra, hence~(iii) holds.

  Now, assume~(iii) holds. If $Q$ is split, then it contains an
  $F$-algebra isomorphic to $F\times F$, so~(ii) holds. For the rest
  of the proof, we assume $Q$ is a division algebra. Let $n_Q$ be the
  norm form of $Q$, which is a $2$-fold Pfister quadratic form in
  $I^2_q(K)$ such that $e_2(n_Q)=Q$ in $\Br(K)$. Since $n_Q$
  represents~$1$, the transfer $s_*(n_Q)$ is isotropic, hence
  Witt-equivalent to a $6$-dimensional nonsingular quadratic form
  $\varphi$ in $I^2_q(F)$. This form satisfies
  $e_2(\varphi)=\Cor_{K/F}(Q)$ in $\Br(F)$, hence $\varphi$ is an
  Albert form of $\Cor_{K/F}(Q)$ as per the definition in
  \cite[(16.3)]{BOI}. In particular, since $\Cor_{K/F}(Q)$ is not a
  division algebra, $\varphi$ is isotropic by \cite[(16.5)]{BOI},
  and therefore $\Wind\bigl(s_*(n_Q)\bigr)\geq2$. By Theorem~\ref{P:}
  there exists a nonsingular quadratic form $\psi$ over $F$ with
  $\dim\psi=2$ such that $\psi_K$ is a subform of
  $n_Q$. Since $Q$ is a division algebra, we have that $\psi_K$
  is anisotropic, hence $\psi$ is similar to the norm form of a unique
  separable quadratic field extension $L/F$. The field $L$ is linearly
  disjoint from $K$ over $F$ because $\psi_K$ is anisotropic. On the
  other hand, $\psi_{KL}$ is hyperbolic, hence $KL$ splits the form
  $n_Q$, and it follows that there exists a $K$-algebra embedding of
  $KL$ in $Q$. Therefore,~(ii) holds. 
\end{proof}

If $K$ is a purely inseparable quadratic extension of $F$, all the
statements of Theorem~\ref{thm:main} hold for every quaternion algebra
over $K$. To see this, recall from the definition of $\Cor_{K/F}$ in
\cite[Rem.~6.9.2]{GS} that the corestriction of every quaternion
$K$-algebra is split. Moreover, if $Q=(M/K,b)$ with $M$ a separable
quadratic extension of $K$, then the separable closure of $F$ in $M$
is a separable quadratic extension of $F$ contained in $Q$ and
linearly disjoint from~$K$.

\section{The Albert form of a corestriction} 
\label{sec:Albert}

Let $Q$ be a quaternion algebra over a separable quadratic field
extension $K$ of an arbitrary field $F$. By definition (see
\cite[(16.3)]{BOI}), the Albert forms of $\Cor_{K/F}Q$ are the
$6$-dimensional nonsingular quadratic forms in $I^2_q(F)$ such that
$e_2(\varphi)=\Cor_{K/F}Q$ in $\Br_2(F)$; they are all similar. As
observed in the proof of Theorem~\ref{thm:main}, an Albert form of
$\Cor_{K/F}Q$ may be obtained from the Witt class of the
($8$-dimensional) transfer
$s_*(n_Q)$ of the norm form of $Q$ for an arbitrary nonzero $F$-linear
functional $s\colon K\to F$ such that $s(1)=0$. In this section, we
sketch a more explicit construction of an Albert form of
$\Cor_{K/F}Q$, inspired by Knus's proof of the Albert--Draxl Theorem
in \cite{Knus}, and we use it to give an alternative proof of
Theorem~\ref{thm:main}. 

We first recall the construction of the corestriction $\Cor_{K/F}Q$.
Let $\gamma$ be the nontrivial $F$-automorphism of $K$ and let
$^\gamma Q$ denote the conjugate quaternion algebra $^\gamma
Q=\{{}^\gamma x\mid x\in Q\}$ with the operations
\[
{}^\gamma x +{}^\gamma y={}^\gamma(x+y),\quad
{}^\gamma x \cdot{}^\gamma y= {}^\gamma(xy),\quad
\lambda\cdot{}^\gamma x = {}^\gamma(\gamma(\lambda)x)
\]
for $x$, $y\in Q$ and $\lambda\in K$. The algebra ${}^\gamma
Q\otimes_KQ$ carries a $\gamma$-semilinear automorphism $s$ defined by
\[
s({}^\gamma x\otimes y)={}^\gamma y\otimes x \qquad\text{for $x$,
  $y\in Q$.}
\]
By definition, the corestriction (or norm) $\Cor_{K/F}(Q)$ is the
$F$-algebra of fixed points (see \cite[(3.12)]{BOI}):
\[
\Cor_{K/F}(Q)=\bigl({}^\gamma Q\otimes_KQ)^s.
\]

Let $\Trd$ and $\Nrd$ denote the reduced trace and the reduced norm on
$Q$. Let also $\sigma$ be the canonical (conjugation) involution on
$Q$. Consider the following $K$-subspace of $^\gamma Q\otimes_KQ$:
\[
V=\{{}^\gamma x_1\otimes1-1\otimes x_2\mid x_1, x_2\in Q\text{ and }
\gamma\bigl(\Trd(x_1)\bigr)=\Trd(x_2)\}.
\]
This $K$-vector space has dimension~$6$ and is preserved by $s$, and
one can show that the $F$-space of $s$-invariant elements has the
following description, where $T_{K/F}\colon K\to F$ is the trace form:
\[
V^s=\{{}^\gamma y\otimes1+1\otimes y\mid y\in Q\text{ and }
T_{K/F}(\Trd(y))=0\}. 
\]
Now, pick an element $\kappa\in K^\times$ such that
$\gamma(\kappa)=-\kappa$. (If $\car F=2$ we may pick $\kappa=1$.) The
following formula defines a quadratic form $\varphi\colon V^s\to F$:
for $y\in Q$ such that $T_{K/F}(\Trd(y))=0$, let
\[
\varphi({}^\gamma y\otimes1+1\otimes y) = \kappa\cdot\bigl(
  \gamma(\Nrd(y))-\Nrd(y)\bigr).
\]
Nonsingularity of the
form $\varphi$ is easily checked after scalar extension to an
algebraic closure of $F$, and computation shows that the linear map
\[
f\colon V^s\to M_2\bigl(\Cor_{K/F}(Q)\bigr) \quad\text{given by}\quad
  \xi\mapsto
  \begin{pmatrix}
    0&\kappa \cdot(\sigma\otimes\id)(\xi)\\
    \xi&0
  \end{pmatrix}
\]
satisfies $f(\xi)^2=\varphi(\xi)$ for all $\xi\in V^s$. Therefore, $f$
induces an $F$-algebra homomorphism $f_*$ defined on the Clifford
algebra $C(V^s,\varphi)$. Dimension count shows that $f_*$ is an
isomorphism
\begin{equation}
  \label{eq:Clif}
  f_*\colon C(V^s,\varphi)\xrightarrow{\sim} M_2(\Cor_{K/F}Q).
\end{equation}
The restriction to the even Clifford algebra is an isomorphism
$C_0(V^s,\varphi) \simeq (\Cor_{K/F}Q)\times (\Cor_{K/F}Q)$, hence the
discriminant (or Arf invariant) of $\varphi$ is trivial. This means
$\varphi\in I^2_q(F)$, and \eqref{eq:Clif} shows that
$e_2(\varphi)=\Cor_{K/F}Q$ in $\Br_2(F)$, so $\varphi$ is an Albert
form of $\Cor_{K/F}Q$.
\medbreak

We use the Albert form $\varphi$ to sketch an alternative proof of
Theorem~\ref{thm:main}. Since all the conditions in
Theorem~\ref{thm:main} trivially hold if $Q$ is split, we may assume
$Q$ is a division algebra. In particular, the base field $F$ is infinite.

Suppose condition~(i) of Theorem~\ref{thm:main} holds. If $x\in
Q$ generates a quadratic $F$-algebra disjoint from $K$, then
$\Trd(x)\in F$ and $\Nrd(x)\in F$ (and $x\notin K$), hence
${}^\gamma(\kappa x)\otimes 1+1\otimes(\kappa x)\in V^s$ is an
isotropic vector of $\varphi$. Since $\varphi$ is an Albert form of
$\Cor_{K/F}Q$, it follows that $\Cor_{K/F}Q$ is not a division
algebra. Therefore, (i) implies (iii).

For the converse, suppose~(iii) holds, and let ${}^\gamma
y\otimes1+1\otimes y\in V^s$ be an isotropic vector for $\varphi$. Then
$\gamma(\Nrd(y))-\Nrd(y)=0$, hence $\Nrd(y)\in F$, and
$T_{K/F}(\Trd(y))=0$. Assuming $y\in K$ quickly leads to a
contradiction, and a density argument shows that we may find such an
element $y$ with $\Trd(y)\neq0$. Then $\kappa y\in Q$ satisfies
\[
\Trd(\kappa y)\in F,\qquad \Trd(\kappa y)\neq0, \quad\text{and}\quad
\Nrd(\kappa y)=\kappa^2\Nrd(y)\in F.
\]
Therefore, $\kappa y$ generates a quadratic \'etale $F$-subalgebra
of $Q$ linearly disjoint from $K$, proving that (ii) (hence also (i))
holds.  

\bibliographystyle{plain}

\end{document}